\documentclass[11pt,fleqn]{scrartcl}

\usepackage[english]{babel}
\usepackage{amsmath,amsfonts, amsthm,xcolor,amssymb}
\usepackage{paralist}
\numberwithin{equation}{section}
\usepackage{mathtools, mathabx}
\usepackage[colorlinks=true,urlcolor=blue, citecolor=blue,linkcolor=blue,linktocpage,pdfpagelabels, bookmarksnumbered,bookmarksopen]{hyperref}

\usepackage[hyperpageref]{backref}

\usepackage{xcolor}

\newtheorem{thm}{Theorem}[section]
\newtheorem{lem}[thm]{Lemma}

\newtheorem{prop}[thm]{Proposition}

\newtheorem{defn}[thm]{Definition}
\theoremstyle{definition}
\newtheorem{rem}[thm]{Remark}
\theoremstyle{remark}

\newcommand{\ds}{\displaystyle}

\newcommand{\R}{\mathbb{R}}
\newcommand{\N}{\mathbb{N}}

\newcommand{\de}{\partial}
\newcommand{\eps}{\varepsilon}

\DeclareMathOperator{\dive}{div}

\DeclareMathOperator{\sign}{sign}

\newcommand\restr[2]{{
  \left.\kern-\nulldelimiterspace 
  #1 
  \vphantom{ |} 
  \right|_{#2} 
  }}
{\left\{\begin{array}{@{}l@{}}}{\end{array}\right.}
\makeatletter 
\def\@makefnmark{} 
\makeatother 

\title{Existence and regularity results for a class of non-uniformly elliptic
Robin problems}

\author{Francesco Della Pietra$^*$, Giuseppina di Blasio$^{**}$, Teresa Radice$^*$
\thanks{
Dipartimento di Matematica e Applicazioni ``R. Caccioppoli'', Universit\`a degli studi di Napoli Federico II, Via Cintia, Monte S. Angelo - 80126 Napoli, Italia. Email: f.dellapietra@unina.it, teresa.radice@unina.it \\
$^{**}$Dipartimento di Matematica e Fisica, Universit\`a degli Studi della Campania ``Luigi Vanvitelli'', viale Lincoln, 5 -81100 Caserta, Italia. Email: giuseppina.diblasio@unicampania.it}
}

\begin{document}

\maketitle
\begin{abstract}
\noindent{\textsc{Abstract.}}
In this paper, we study the existence and the summability of solutions to a Robin boundary value problem whose prototype is the following:
\begin{equation*}
\begin{cases}
-\dive(b(|u|)\nabla u)=f &\text{in }\Omega,\\[.2cm]
\ds\frac{\de u}{\de \nu}+\beta u=0 &\text{on }\de\Omega
\end{cases}
\end{equation*}
where $\Omega$ is a bounded Lipschitz domain in $\R^N$, $N>2$, $\beta>0$, $b(s)$ is a positive function which may vanish at infinity and $f$ belongs to a suitable Lebesgue space. The presence of such a function $b$ in the principal part of the operator prevents it from being uniformly elliptic when $u$ is large.
\\

\end{abstract}

\begin{center}
\begin{minipage}{13cm}
\small
\tableofcontents
\end{minipage}
\end{center}

\section{Introduction}
In this paper, we are exploring the existence of solutions to the following Robin boundary value problem:
\begin{equation}
\label{maing}
\begin{cases}
-\dive(A(x,u)\nabla u)=f &\text{in }\Omega,\\[.2cm]
\ds\frac{\de u}{\de \nu}+\beta u=0 &\text{on }\de\Omega
\end{cases}
\end{equation}
where $\Omega$ is a bounded Lipschitz domain in $\R^N$, $N>2$, and $A(x,t)\colon  \overline\Omega\times\R\to \R$ is a bounded Carath\'eodory function such that
\begin{equation}
\label{condA}
     A(x,t) \ge b(|t|),\qquad \text{for a.e. }x\in \overline\Omega,\; \forall t\in\R,
\end{equation}
$b\colon [0,+\infty[\to ]0,+\infty[$ is a bounded continuous function, $\beta>0$, and $\nu$ is the outer normal to $\Omega$ on $\de \Omega$. Moreover, $f\in L^p(\Omega)$, for some $p>1$. Our aim is to prove a priori estimates and existence results for \eqref{maing} under various assumptions on the summability of $f$. 

The main difficulty in solving problem \eqref{maing} is that the ellipticity condition, given in terms of the function $b$, may prevent the operator from being coercive when the solution, represented by $u$, is unbounded. In simpler terms, the operator doesn't meet the conditions needed to guarantee the existence of a solution using classical methods even if the datum has the right summability. However, if $u$ is bounded, the operator becomes coercive, and we can then use classical theory to prove the existence of a weak solution. In this case, the results we obtain are similar to those for the Laplacian equation (with $\theta=0$). 

The model problem of the general setting considered above is given by
\begin{equation}\label{intro}
\begin{cases}
-\dive(b(|u|)\nabla u)=f &\text{in }\Omega,\\[.2cm]
\ds\frac{\de u}{\de \nu}+\beta u=0 &\text{on }\de\Omega
\end{cases}
\end{equation}
and the coefficient $b(|u|)$ is allowed to go to zero when $u$ becomes large.

This kind of problem has been widely studied under Dirichlet homogeneous boundary conditions ($\beta=+\infty$). In this case, it is well–known that the existence of solutions is strictly related to the summability of $b$ (we refer the reader, for example, to \cite{2,1,bg,boc2006,bocbrez,10,mena,gp}, and \cite{18,37,do,cina,merc,s} with additional terms in the equation, or \cite{19,31} with different hypotheses on $b$). 
Also, some existence results have been proved in the evolution (Cauchy-Dirichlet) problem. See, for example, to \cite{pp,dpdb}, while in \cite{17} the case with additional gradient–dependent terms in the equation is addressed.

As far as we know, the degenerate problem with Robin boundary condition has not been studied as well. To fill the gap, we aim to obtain a priori estimates and existence results, enhancing the differences with the Dirichlet case.


First of all, let us recall the definition of solution for \eqref{maing}.
\begin{defn}
We say that a function $u\in H^1(\Omega)$ is a weak solution of \eqref{maing} if 
\begin{equation*}
\int_\Omega A(x,u)\nabla u\cdot \nabla \varphi dx +\beta \int_{\de\Omega}
A(x,u) u\varphi \, \, d \mathcal H^{N-1} =\int_\Omega f \varphi dx,    
\end{equation*}
for all $\varphi\in  H^1(\Omega)$.
\end{defn}

As already mentioned, we need to give extra assumptions on $b$ in order to get existence results. A main role is played by the function
\begin{equation}
    \label{B}
B(t):=\int_0^t b(s)ds, \quad t\ge 0.
\end{equation}

Our first main theorem regards the existence of bounded solutions.

\begin{thm}\label{bounded}
Let $\Omega$ be a Lipschitz bounded open set in $\R^N$.
Suppose that $\sup B=+\infty$ and 
\begin{equation}
    \label{condizione_bordo}
    tb(t)\ge \Gamma B(t),\quad\forall t\ge 0,
\end{equation}
for some $\Gamma>0$. If $f\in L^q(\Omega)$, with  $q>\frac{N}{2}$, then there exists a weak solution $u\in  H^{1}(\Omega)\cap L^{\infty}(\Omega)$ of \eqref{maing}.
\end{thm}
\begin{rem}
We emphasize that in the Dirichlet case, a general existence result of bounded solution with $f\in L^{\frac{N}{2}+\eps}(\Omega)$ can be obtained just by assuming that $\sup B=+\infty$, that is $b\not\in L^1(0,+\infty)$ (\cite{2,1,bocbrez,10}). Actually, in the Robin case the condition \eqref{condizione_bordo} is necessary to have the existence of bounded solutions (see Section \ref{esempiomain}).
\end{rem}

In order to study the existence of an unbounded solution we reduce ourselves to the case 
\begin{equation}
    \label{b(u)}
    b(s)=\frac{1}{(1+s)^{\vartheta}}, \quad s\ge 0, \quad \text{for some }0<\vartheta\leq 1.
\end{equation}
\begin{rem}
If we consider $b(s)$ as in \eqref{b(u)}, the condition \eqref{condizione_bordo} is satisfied if $0<\vartheta<1$ (with $\Gamma=1-\vartheta$), otherwise, in the case $\vartheta=1$, it is false (see again the example in Section \ref{esempiomain}). 
\end{rem}

\begin{thm}
\label{unbouded}
Let $\Omega$ be a Lipschitz bounded open set in $\R^N$. Let us consider problem \eqref{maing} with \eqref{condA} and $b(u)$ as in \eqref{b(u)}.
If  $f\in L^q(\Omega)$, with 
\begin{equation}
\label{condizioneq}
    \frac{2N}{N+2-\vartheta(N-2)}\leq q<\frac{N}{2},
\end{equation} 
then there exists a weak solution $u\in H^1(\Omega)\cap L^{q^{**}(1-\vartheta)}(\Omega) $ of {\color{magenta}\eqref{maing}}, where $q^{**}=\frac{qN}{N-2q}$.
\end{thm}
The previous result assures that the solution $u$ belongs to the energy space $H^1$, which cannot be directly derived from the equation since $u$ may be unbounded.

If one decreases the summability of the datum, the solutions do not necessarily belong anymore to the energy space. More precisely we get the following result.
\begin{thm}
 \label{special}
Let $\Omega$ be a Lipschitz bounded open set in $\R^N$. Let us consider problem \eqref{maing} with \eqref{condA} and $b(u)$ as in \eqref{b(u)}. If  $f\in L^q(\Omega)$ with 
\begin{equation}
\label{condizioneq2}
  \frac{2N-N\vartheta}{N+2-N\vartheta}  \leq q<\frac{2N}{N+2-\vartheta(N-2)},
\end{equation} 
then there exists a solution $u\in W^{1,s}(\Omega)$ of \eqref{maing}, where $s=\frac{2N-N\vartheta}{N-\vartheta}<2$ in the sense that
\begin{equation}
\label{streuza}
\int_\Omega A(x,u)\nabla u\cdot \nabla \varphi dx +\beta \int_{\de\Omega}
A(x,u){u\varphi} d \mathcal H^{N-1} =\int_\Omega f \varphi dx,   
\end{equation}
for all $\varphi\in  W^{1,s'}(\Omega)$.
 \end{thm}

The paper is organized as follows: Section 2 contains some definitions and properties that will be used in the next sections; in Section 3 we prove our results, distinguishing the cases of bounded and unbounded solutions. Finally, Section 4 is devoted to analysing an example of our problem.

\section{Preliminaries}

We start by recalling the definitions of the Marcinkiewicz spaces and their properties (see \cite{hunt} for more details).

Let us assume that $\Omega$ is an open set in $\R^N$, $N\geq2$. Given $1<p<+\infty$ the  Marcinkiewicz  space $L^{p,\infty}(\Omega)$ consists in all measurable functions $f$ defined on $\Omega$ such that
\begin{equation*}\label{2.1}
\|f\|^p_{p,\infty}=\sup_{t>0}t^{ p}\mu_f(t)<+\infty
\end{equation*}
where $\mu_f(t)=\left| \left\{ x\in \Omega: |f(x)|>t \right\}\right|$ is the distribution function of $f$.
It holds that when $\Omega$ is bounded
\begin{equation*}\label{LorLebEmb}
L^p(\Omega) \subset L^{p,\infty}(\Omega)\subset L^r(\Omega),
\end{equation*}
whenever $1< r<p< \infty.$

We recall the following version of the H\"older inequality proved in \cite[Proposition A.1]{dos1}.
\begin{prop}
\label{holder}
	Let $1<p, p'<\infty$ be such that $\displaystyle \frac1{p}+\frac1{p'}=1$.
	Assume that $f_1, f_2\>:\>\Omega\to\mathbb R$ and $g_1, g_2\>:\>\partial\Omega\to\mathbb R$ are measurable functions satisfying
	$f_1\in L^p(\Omega)$, $f_2\in L^{p'}(\Omega)$, $g_1\in L^p(\partial\Omega, \lambda)$ and $g_2\in L^{p'}(\partial\Omega, \lambda)$. Then $f_1f_2\in L^1(\Omega)$, $g_1g_2\in L^1(\partial\Omega, \lambda)$ and
	\begin{multline*}
		\int_\Omega|f_1f_2|\, dx+\int_{\partial\Omega}\lambda(x)|g_1g_2|\, d\mathcal H^{N-1}\\
		\le \left[\int_\Omega|f_1|^p\, dx+\int_{\partial\Omega}\lambda(x)|g_1|^p\, d\mathcal H^{N-1}\right]^{\frac1p}\left[\int_\Omega|f_2|^{p'}\, dx+\int_{\partial\Omega}\lambda(x)|g_2|^{p'}\, d\mathcal H^{N-1}\right]^{\frac1{p'}}.
	\end{multline*}
\end{prop}

Let $\Omega$ be a bounded Lipschitz domain, we recall that for any $\alpha>0$ there exists a positive constant $\lambda(\Omega,\alpha,N)$ such that, for any $u\in H^1(\Omega)$
\[
\int_\Omega u^2dx \le  \frac{1}{\lambda(\Omega,\alpha,N) }\left[ \int_\Omega |\nabla u|^2 dx+ \alpha \int_{\de\Omega} u^2 d \mathcal H^{N-1} 
\right].
\]
As well known, the optimal constant $\lambda(\Omega,\alpha, N)$ in the previous inequality is the first Robin eigenvalue of the Laplacian. Combining such inequality with the Sobolev inequality, it holds that
\begin{equation}
\label{sob_eigen}
 \|u\|_{L^{2^*}(\Omega)}^2    \le C (\Omega,\alpha,N) \left[
    \int_\Omega |\nabla u|^2 dx+ \alpha \int_{\de\Omega} |u|^2 d \mathcal H^{N-1}
    \right]
 \end{equation}
    where $2^*=\frac{2N}{N-2}$.

In the paper we will also make use of the following well-known trace inequality. If $\Omega$ is a bounded Lipschitz open set in $\R^N$, there exists a positive constant  $C>0$ such that when $1<p<N$
\begin{equation}
\label{traceembeddq}
\|u\|_{L^t(\partial \Omega)} \leq C\|u\|_{W^{1, p}(\Omega)} \text {  , for } t=\frac{p(N-1)}{N-p}
\end{equation}
with a compact embedding (see \cite[Chapter 2, Theorems 4.2 and 6.2]{necas}). 

Finally, if $u\in W^{1,p}(\Omega)$, and $\Omega$ is a bounded Lipschitz open set in $\R^N$, then the Poincar\'e-Wirtinger inequality holds:
\begin{equation*}
    \|u-u_\Omega\|_{L^p(\Omega)} \le C\|\nabla u\|_{L^p(\Omega)},
\end{equation*}
with $u_\Omega=\frac{1}{|\Omega|}\int_\Omega u\, dx$.

This obviously implies that
\begin{equation}
\label{poinc}
    \|u\|_{L^p(\Omega)} \le C\left(\|\nabla u\|_{L^p(\Omega)}+\|u\|_{L^1(\Omega)}\right).
\end{equation}

\section{Proof of the main results}
In this section, we prove the existence results under various hypotheses on the datum $f$. To this end, we will proceed by approximation, so we consider the following approximated problems 
\begin{equation}
\label{maingapprox}
\begin{cases}
-\dive(A_n(x,u_n)\nabla u_n)=f &\text{in }\Omega,\\[.2cm]
\ds\frac{\de u_n}{\de \nu}+\beta u_n=0 &\text{on }\de\Omega,
\end{cases}
\end{equation}
where
\[
A_n(x,t)=A(x,T_n(t)),
\]
 and 
 \[
 T_n(t)=
 \begin{cases}
     t&\text{if }|t|\le n,\\
     n\sign t &\text{if }|t|> n,
 \end{cases}
 \]
 denotes the usual truncation function at level $n>0$. We observe that the operator in \eqref{maingapprox} is now coercive, so the existence of bounded solutions is guaranteed by classical arguments. Then we get apriori estimates for this solutions, that together with the boundedness and the continuity of $A_n$ will allow us to pass to the limit in \eqref{maingapprox}, finding a solution of \eqref{maing}.

\subsection{Bounded solutions}
We start considering the case in which the datum $f$ has high summability.
\begin{lem}
\label{stimainfinito}
In the hypotheses of Theorem \ref{bounded}, if $f\in L^q(\Omega)$, $q>\frac{N}{2}$, and $u_n\in H^1(\Omega)\cap L^{\infty}(\Omega)$ is a weak solution of \eqref{maingapprox}, then
\begin{equation}
    \label{stima1}
     B(u_n(x)) \le K, \quad \quad \forall x\in \Omega
\end{equation}
where  $B$ is defined in \eqref{B} and  $K$ is a positive constant independent on $n$.
\end{lem}

\begin{proof}
Let $u_n$ be a weak solution of \eqref{maingapprox}, we consider $G_k(t)=(|t|-k)^+\sign t=t-T_k(t)$, for $k>0$ and we take $\varphi_k=G_k(B(u_n))$ as test function in \eqref{maingapprox}. Denoting $A_k=\{|u_n|>k\}$, by  \eqref{condA} it holds
\[
\int_{A_k} \left| \nabla G_k (B(u_n)) \right|^2 dx +\beta \int_{\de A_k} u_n b(u_n) G_k(B(u_n)) d\mathcal{H}^{N-1} \le \int_{A_k} f G_k(B(u_n))dx.
\] 
We explicitly observe that $\varphi_k$ vanishes also on $\de A_k\cap\Omega$.
Moreover, by \eqref{condizione_bordo} for some $\Gamma>0$, we get
\[
u_n b(u_n) G_k(B(u_n)) \ge \Gamma B(u_n) G_k(B(u_n)) \ge \Gamma  G_k(B(u_n))^2,
\]
and then, by using the H\"older inequality, 
\begin{multline*}
\int_{A_k} \left| \nabla G_k (B(u_n)) \right|^2 dx +\beta \Gamma \int_{\de A_k} G_k(B(u_n))^2 d\mathcal{H}^{N-1} \le \\ \le  \|f\|_{L^{\frac{2N}{N+2}}(A_{k})} \left( \int_{A_k} G_k(B(u_n))^{2^*}dx\right)^{\frac{1}{{2^{*}}}}.
\end{multline*}
Now, using the  inequality
\eqref{sob_eigen}, it holds
\[
\|G_k(B(u_n))\|_{L^{2^{*}}(A_{k})} \le C \|f\|_{L^{\frac{2N}{N+2}}(A_{k})}.
\]
Then by the H\"older inequality it holds that
\[
\|G_k(B(u_n))\|_{L^{2^{*}}(A_{k})} \le C \|f\|_{L^{q}(\Omega)} |A_{k}|^{\left(1-\frac{2N}{q(N+2)}\right)\frac{N+2}{2N}}.
\]
Hence, if $h>k>0$ then $G_{k}(B(u_{n}))\geq h-k$ on $A_{h}$, and finally, being also $q>\frac{N}{2}$, it holds that
\[
|A_{h}| \le \frac{C}{(h-k)^{2^{*}}} |A_{k}|^{1+\eps}, 
\]
for some $\eps>0$. The thesis follows by applying the classical Stampacchia Lemma (see \cite[Lemma 4.1]{st}).
\end{proof}

\paragraph{Proof of Theorem \ref{bounded}}
\begin{proof}
Let $u_{n}\in H^{1}(\Omega)$ be a weak solution of the approximating problem \eqref{maingapprox}. Then, being $B$ increasing and unbounded, by \eqref{stima1}
\[
u_{n}(x) \le B^{-1}(K), \quad \forall x\in \Omega, \forall n\in\N.
\]
Hence, trivially, for $n$ sufficiently large as $n>B^{-1}(K)$, it holds that $A_n(x,u_n)=A(x,u_n)$  and than $u_n$ is a bounded solution of  problem \eqref{maing}.
\end{proof}

\subsection{Unbounded solutions}
\begin{lem}
\label{stimaLq}
Under the hypotheses of Theorem \ref{unbouded}, if $u_n\in H^1(\Omega)$ is a weak solution of problem
\begin{equation}
\label{maingapprox_2}
\begin{cases}
-\dive(A(x,u_n)\nabla u_n)=f_n &\text{in }\Omega,\\[.2cm]
\ds\frac{\de u_n}{\de \nu}+\beta u_n=0 &\text{on }\de\Omega
\end{cases}
\end{equation}
 with $f_n=T_n(f)$, then
\begin{equation*}
        \label{stimaLq}
    \|u_n\|_{L^{q^{**}(1-\vartheta)}(\Omega)} \le K_1,
\end{equation*}
with $q$ as in \eqref{condizioneq}, and 
\[
\|u_n\|_{H^1(\Omega)} \le K_2,
\]
where $K_1,K_2$ are  positive constants independent on $n$, which depend only by $\vartheta$, $q$, $N$, $|\Omega|$, $\|f\|_{L^q(\Omega)}$.
\end{lem}
\begin{proof}
As in Lemma \ref{stimainfinito}, there exists a weak solution $u_n\in H^1(\Omega)\cap L^\infty(\Omega)$  of \eqref{maingapprox_2}. We consider $\varphi_n=[(1+|u_n|)^p-1]\sign u_n$, for a suitable $p>1$, as test function in \eqref{maingapprox_2}. Using \eqref{condA}, it follows that
\begin{multline}
\label{disuguaglianza2}
C\int_{\Omega} \left| \nabla (1+|u_n|)^{\frac{p+1-\vartheta}{2}} \right|^2 dx +\beta \int_{\de \Omega} |u_n| (1+|u_n|)^{-\vartheta}((1+|u_n|)^p-1) d\mathcal{H}^{N-1} \\ \le \int_{\Omega} |f_n| ((1+|u_n|)^p-1) dx,
\end{multline}
with $C=p\left(\frac{2}{p+1-\vartheta}\right)^2$. Now, we observe that
\begin{equation}
\label{diseg}
|u_n|(1+|u_n|)^{-\vartheta}((1+|u_n|)^p-1)\ge \frac{1}{p+1}\left((1+|u_n|)^{\frac{p+1-\vartheta}{2}}-1\right)^2.    
\end{equation}
Indeed the function
\[
h(t)= (t-1)t^{-\vartheta}(t^p-1)-\frac{1}{p+1} \left(t^{\frac{p+1-\vartheta}{2}}-1\right)^2
\]
is such that $h(1)=0$ and increasing as $t\ge 1$ since
\[
h'(t)= t^{\vartheta-1}[pt^{p-\vartheta}(t-1)+g(t)]
\]
 with 
 \[
 g(t)=\frac{p+1+\vartheta}{p+1} t^{\frac{p+1-\vartheta}{2}}-t^{1-\vartheta}-\frac{\vartheta}{p+1}
 \]
 such that $g(1)=0$ and $g'(1)\geq 0$.

Using inequality \eqref{diseg} in \eqref{disuguaglianza2}, we get
\begin{multline} \label{intero}
C\int_{\Omega} \left| \nabla (1+|u_n|)^{\frac{p+1-\vartheta}{2}} -1\right|^2 dx +\beta \int_{\de \Omega}  ((1+|u_n|)^{\frac{p+1-\vartheta}{2}} -1)^2 d\mathcal{H}^{N-1} \\ \le \int_{\Omega} |f_n| ((1+|u_n|)^p-1) dx,
\end{multline}

by the inequality \eqref{sob_eigen} we get
\[
\|u_n\|^{p+1-\vartheta}_{L^{(p+1-\vartheta)\frac{2^*}{2}}(\Omega)} \le C  \int_{\Omega} |f_n| |u_n|^p dx.
\]
 We choose $p$ such that $pq'=(p+1-\vartheta)\frac{2^*}{2}$ that means $p=\frac{(1-\vartheta)N(q-1)}{N-2q}$, by the H\"older inequality we have
\begin{equation}
\label{Lqstima}
\|u_n\|^{1-\vartheta}_{L^{(p+1-\vartheta)\frac{2^*}{2}}(\Omega)} \le C \|f_n\|_{L^{q}(\Omega)},    
\end{equation}
and
\begin{equation}
\label{gradstima}
\int_\Omega\left|\nabla u_n\right|^2\left(1+\left|u_n\right|\right)^{p-1-\vartheta} \le C.    
\end{equation}
Now, we notice that $(p+1-\vartheta)\frac{2^*}{2}=q^{**}(1-\vartheta)$ and condition \eqref{condizioneq} is equivalent to require that $p\ge 1+\vartheta$. Then \eqref{gradstima} gives a bound of the gradient of $u_n$ in $L^2$.
Moreover, \eqref{intero} together with \eqref{Lqstima}, it implies that
\[
\|u_n\|_{L^{p+1-\vartheta}(\de\Omega)}\leq C
\]
with $p+1-\vartheta>2$. Finally, we have
\[
\|u_n\|_{H^1(\Omega)} \le C.
\]
\end{proof}

\paragraph{Proof of Theorem \ref{unbouded}}
\begin{proof}
Let $u_{n}\in H^{1}(\Omega)$ be a weak solution of the approximating problem \eqref{maingapprox_2} with $f=f_n$. By Lemma \ref{stimaLq} there exists a subsequence still denoted by $u_n$ such that 
\[
u_n \rightarrow u \text{ in } L^2(\Omega),
\]
\[
u_n \rightharpoonup u \text{ in } H^1(\Omega),
\]
\[
u_n \rightarrow u \text{ a.e. in } \Omega.
\]

By the compact trace embedding \eqref{traceembeddq} with $p=2$, we get that $u_n$ strongly converges in $L^2(\de\Omega)$. Being $A$ bounded and continuous, this allows to pass to the limit and conclude that $u$ is a weak solution of \eqref{maing}.
\end{proof}

\begin{lem}
 \label{stimaLs}
In the hypotheses of Theorem \ref{special}, if $u_n\in H^1(\Omega)$ is  a weak solution of \eqref{maingapprox_2} and $f\in L^q(\Omega)$, with $q$ as in \eqref{condizioneq2}, then
\begin{equation*}
        \label{stimagradd}
        \| u_{n}\|_{W^{1,s}(\Omega)}\leq K
\end{equation*}
with $s=\frac{2N-N\vartheta}{N-\vartheta}$
where $K$ is a positive constant independent on $n$, which depends only by $\vartheta,q,N,|\Omega|$ and $\|f\|_{L^q(\Omega)}$.
 \end{lem}
\begin{proof}
Let $u_n$ be a solution of \eqref{maingapprox_2}. Using the Poincar\'e-Wirtinger inequality \eqref{poinc}, the H\"older inequality \eqref{holder}, and using $u_{n}$ as test function in \eqref{maingapprox_2}, we obtain
\begin{multline} \label{1}
\|u_n\|_{W^{1,s}}^s \le C
\int_{\Omega}(|\nabla u_{n}|^{s} +|u_{n}|)dx \le \\
C2^{\frac{2-s}{s}}\left(
\int_{\Omega} \frac{|\nabla u_{n}|^{2}}{(1+|T_n(u_{n})|)^{\vartheta}}dx+\int_{\Omega} \frac{| u_{n}|^{\frac2s}} {(1+|T_n(u_{n})|)^{\vartheta}}dx\right)^{\frac{s}{2}} \times \\  \qquad\qquad \times
\left(\int_{\Omega} (1+|T_n(u_{n})|)^{\frac{\vartheta	s}{2-s}}dx
\right)^{\frac{2-s}{s}} \le
\\
\le 
C2^{\frac{2-s}{s}}\left(
\int_{\Omega}|f_{n}||u_{n}|dx -\int_{\de\Omega} A(x,u_n)u_n^2 d\mathcal H^{N-1} +\int_{\Omega} \frac{| u_{n}|^{\frac2s}}{(1+|T_n(u_{n})|)^{\vartheta}}dx\right)^{\frac{s}{2}} \times \\ \qquad\qquad \times
\left(\int_{\Omega} (1+|T_n(u_{n})|)^{\frac{\vartheta	s}{2-s}}dx
\right)^{\frac{2-s}{s}}\le \\
\le
C2^{\frac{2-s}{s}}\left(
\int_{\Omega}|f_{n}||u_{n}|dx +\int_{\Omega} {(1+|u_{n}|)^{\frac2s-\vartheta}}dx\right)^{\frac{s}{2}}
\left(\int_{\Omega} (1+|u_{n}|)^{\frac{\vartheta	s}{2-s}}dx
\right)^{\frac{2-s}{s}},
 \end{multline}
 since $T_n(s)\leq s$, for $s>0$. 
Now, observing that the Sobolev conjugate is $s^{*}=\frac{\vartheta s}{2-s}$, and that $2-\vartheta<s$, by H\"older and Sobolev inequalities we get
\begin{multline} \label{2}
\left(
\int_{\Omega}|f_{n}||u_{n}|dx +\int_{\Omega} {(1+|u_{n}|)^{\frac2s-\vartheta}}dx\right)^{\frac{s}{2}}
\left(\int_{\Omega} (1+|u_{n}|)^{\frac{\vartheta	s}{2-s}}dx
\right)^{\frac{2-s}{2}} \le  \\ \le
\left(
\|f\|_{L^{q}}\|u_{n}\|_{L^{s^{*}}}+\left(\int_{\Omega} (1+|u_{n}|)^{s^{*}}dx\right)^{\frac{2-\vartheta s}{s^{*}s}}
\right)^{\frac{s}{2}}\left(\int_{\Omega}(1+|u_{n}|)^{s^{*}}dx\right)^{\frac{s\vartheta}{2s^{*}}} \le \\ \le
C\left(
\|f\|_{L^{q}}\|u_{n}\|_{W^{1,s}}+ \|u_{n}\|_{W^{1,s}}^{\frac2s-\vartheta}
\right)^{\frac{s}{2}}\|u_{n}\|_{W^{1,s}}^{\frac{s\vartheta}{2}}.
\end{multline}
combining \eqref{1} and \eqref{2}, we get the thesis. 
 
\end{proof}
\paragraph{Proof of Theorem \ref{special}}
\begin{proof}
Let $u_{n}$ be solution of the approximating problem \eqref{maingapprox_2} with $f=f_n$. By Lemma \ref{stimaLs} there exists a subsequence still denoted by $u_n$ such that 
\[
u_n \rightarrow u \text{ in } L^s(\Omega),
\]
\[
u_n \rightharpoonup u \text{ in } W^{1,s}(\Omega),
\]
\[
u_n \rightarrow u \text{ a.e. in } \Omega,
\]
with $s=\frac{2N-N\vartheta}{N-\vartheta}$. By the trace embedding \eqref{traceembeddq}, we get that $u_n$, and then $v_n=A(x,T_n(u_n))u_n$, are bounded in $L^r(\de\Omega)$ with  $r= \frac{s(N-1)}{N-s}>s$. In conclusion, we have strong convergence in
$L^s(\de \Omega)$.
This allows to pass to the limit and to get that $u$ is a solution of \eqref{intro} in the sense of \eqref{streuza}.
\end{proof}

\section{An example}
\label{esempiomain}
Let us consider problem \eqref{intro} defined in $\Omega=B_R$, the ball centered at the origin with radius $R$ and with
\[
b(t)= \frac{1}{(1+t)^\vartheta},
\]
where $0<\vartheta\le 1$, and let us take into account a datum $f=\frac{A}{|x|^\gamma}\in L^{\frac{N}{\gamma},\infty}(B_R)$, with $A\ge 0$. We consider for simplicity the case of bounded solutions. Hence we assume that $0\le \gamma<2$. Then by the properties of Marcinkiewicz spaces $f\in L^{\frac{N}{2}+\eps}(B_R)$ for some $\eps>0$.

More precisely, we consider
\begin{equation}
\label{maingex2}
\begin{cases}
-\dive\left(\frac{\nabla u}{(1+|u|)^\vartheta}\right)=\frac{A}{|x|^\gamma} &\text{in }B_R,\\[.2cm]
\ds\frac{\de u}{\de \nu}+\beta u=0 &\text{on }\de B_R.
\end{cases}
\end{equation}
Performing the change of variable $v=B(u)$, then $v$ has to satisfy
\begin{equation}
\label{maingexcv}
\begin{cases}
-\Delta v= \dfrac{A}{|x|^\gamma}   &\text{in }\Omega,\\[.2cm]
\ds\frac{\de v}{\de \nu}+ \beta F(v)=0 &\text{on }\de\Omega 
\end{cases}
\end{equation}
where
\[
F(v)= b(B^{-1}(v)) B^{-1}(v).
\]
It is easy to see that being $F$ monotone increasing, then \eqref{maingexcv} admits at most a solution. Indeed, if $v_1$ and $v_2$ are two solutions in $H^1(B_R)\cap L^\infty(B_R)$ of \eqref{maingexcv}, using $v_1-v_2$ as test function in the weak formulation of the problem, we have
\begin{multline*}
    \int_{B_R} \nabla v_i\cdot \nabla (v_1-v_2) dx+\beta \int_{\de B_R} F(v_i)(v_1-v_2)d\mathcal H^{N-1} =\int_{B_R} f(v_1-v_2)dx,\\ i=1,2.
\end{multline*}  
Then, by subtracting,
\[
\int_{B_R} |\nabla (v_1-v_2)|^2 dx+\beta \int_{\de B_R} (F(v_1)-F(v_2))(v_1-v_2)d\mathcal H^{N-1} =0.
\]
Being $F$ strictly increasing, this equality holds if and only if $v_1\equiv v_2$ in $B_R$.

We look for positive radial solutions of \eqref{maingexcv}. In this case, it reduces to
 \[
 \begin{cases}
\displaystyle -v''-\frac{N-1}{r}v'=\frac{A}{r^\gamma} \quad \text{ in }]0,R[,\\[.2cm]
v'(0)=0,\\[.2cm]
v'(R)+\beta F(v(R))=0.
\end{cases}
 \]
  By an easy computation, the solution has to be
\[
v(r)= \ds v(R)+\frac{A}{N-\gamma}\frac{R^{2-\gamma}-r^{2-\gamma}}{2-\gamma},
\]
where, due to the boundary condition, $v(R)$ must be chosen in such a way that
\begin{equation}
\label{passesempio}
F(v(R))=\frac{A}{N-\gamma}R^{1-\gamma}.    
\end{equation}
We observe that the behavior of problem \eqref{maingexcv}, and then \eqref{maingex2}, depends on $\vartheta$. Hence we distinguish two cases.
\begin{enumerate}
    \item If $0< \vartheta <1$, then
\begin{equation}
    \label{F2}
F(v)= \frac{B^{-1}(v)}{(1+B^{-1}(v))^\vartheta}
\end{equation}
is strictly increasing and unbounded. Hence the condition \eqref{passesempio} is satisfied for a unique value of $v(R)$. 
 Therefore the function 
 \[
    u(r)=B^{-1}(v(r)), \quad r\in [0,R]
    \]
is a solution of \eqref{maingex2}.
  \item If $\vartheta=1$, then \eqref{F2} is bounded and \eqref{passesempio} is satisfied only if a boundedness condition on the right-hand side of \eqref{passesempio} is in force. In other words, if
  \begin{equation}
      \label{condition}
        \frac{A}{N-\gamma}R^{1-\gamma}<\sup_{t\in [0,+\infty[} F(t),
  \end{equation}
  there exists a bounded solution of \eqref{maingexcv} and then of \eqref{maingex2}. On the other hand, the uniqueness assures that \eqref{maingexcv} may admit only a radial solution. Hence, if condition \eqref{condition} is not in force, problem \eqref{maingexcv} (and then of \eqref{maingex2}) does not admit a bounded radial solution.
\end{enumerate}

\section*{Acknowledgements}
This work has been partially supported by PRIN PNRR 2022 ``Linear and Nonlinear PDE’s: New directions and Applications'', by GNAMPA of INdAM, by START within the program of the University "Luigi Vanvitelli" reserved to young researchers, Piano strategico 2021-2023, by the ``Geometric-Analytic Methods for PDEs and Applications" project - funded by European Union - Next Generation EU  within the PRIN 2022 program (D.D. 104 - 02/02/2022 Ministero dell'Universit\`{a} e della Ricerca). This manuscript reflects only the authors' views and opinions and the Ministry cannot be considered responsible for them.

{\small

\bibliographystyle{plain}

\end{document}